\documentclass[11pt]{article}

\usepackage[usenames,dvipsnames,condensed]{xcolor}

\usepackage{amsthm}
\usepackage{amsfonts}
\usepackage[UKenglish]{babel}
\usepackage[usenames]{xcolor}
\usepackage{graphicx}
\usepackage{soul}
\usepackage{stfloats}
\usepackage{morefloats}
\usepackage{cite}
\usepackage{lscape}
\usepackage{epstopdf}
\usepackage{tikz-cd} 

\usepackage{amsmath,amstext,amsopn,amsfonts,eucal,amssymb}
\usepackage{graphicx,wrapfig,url}

\newcommand{\GAP}{\textsf{GAP}}
\newcommand{\Sym}{\mathbb{S}}

\newcommand\Z{{\mathbb Z}}

\newtheorem{theorem}{Theorem}[section]
\newtheorem{corollary}[theorem]{Corollary}
\newtheorem{lemma}[theorem]{Lemma}
\newtheorem{proposition}[theorem]{Proposition}

\newtheorem{remark}[theorem]{Remark}

\newtheorem{question}[theorem]{Question}

\begin{document}

\title{Algebraic Properties  of Quandle  Extensions \\ and  Values of Cocycle Knot Invariants
}

\author{W. Edwin Clark \   and \  Masahico  Saito \\
Department of Mathematics and Statistics\\ University of South Florida
}

\date{\empty}

\maketitle

\begin{abstract}
Quandle 2-cocycles define invariants of classical and virtual knots, and extensions of quandles.
We show that the quandle 2-cocycle invariant with respect to a non-trivial $2$-cocycle
 is constant, or takes some other restricted form,  for classical knots when the corresponding extensions
satisfy certain algebraic conditions. In particular, if an abelian extension is a conjugation quandle, then the corresponding cocycle invariant is constant.
Specific examples are presented from the list of connected quandles of order less than 48. 
Relations among various quandle epimorphisms involved are also examined. 
\\[5mm]
\end{abstract}

\section{Introduction}\label{intro-sec}

Sets with certain self-distributive operations called  {\it quandles}
have been studied since  the 1940s \cite{Taka},
and have been applied to  knot theory since early 1980s  \cite{Joyce,Mat}.
The number of  colorings of knot diagrams by quandle elements, in particular, 
has been widely used as a
knot invariant. Algebraic homology theories for quandles 
were defined \cite{CJKLS,FRS1}, and
 investigated.
 Knot invariants using cocycles have been defined \cite{CJKLS} and applied to knots and knotted surfaces
 \cite{CKS}. 
 Extensions of quandles by cocycles have been studied, for example, in \cite{AG,CENS,Eis3}.

Computations by Vendramin \cite{Leandro} significantly expanded the list
of small connected quandles.
These
quandles, called here {\it Rig} quandles,
may be found in the \GAP~package Rig \cite{rig}.  
Rig includes all  connected quandles of order less than 48, at this time. 
Properties of some of Rig quandles,  such as homology groups and cocycle invariants,
are also found in  \cite{rig}.
We  use the notation $Q(n,i)$
for the $i$-th quandle of order $n$ in the list of Rig quandles.

It was observed that some  Rig quandles have non-trivial second cohomology, yet have 
constant $2$-cocycle invariants with non-trivial $2$-cocycles,
as much as computer calculations have been performed for the knot table (Remark 4.5, \cite{CSV}).
It does not seem to have been 
established previously  whether they actually have constant values for all classical knots.
From  Theorem 5.5 in \cite{CKS:geom},   
any non-trivial $2$-cocycle has non-constant  invariant values for some virtual links.
Thus it is of interest if these quandles actually have constant values for all classical knots.
More generally, possible values of the cocycle invariants are largely unknown, and are of interest.
In this paper, we show that  certain algebraic properties of quandles imply 
that  the cocycle invariant is constant, or takes some restricted form, for classical knots.
In particular, we prove that 
several specific Rig quandles,  including some of those 
conjectured in \cite{CSV}, 
have constant cocycle invariant values for all classical knots for some non-trivial $2$-cocycles.

In Section~\ref{sec:prelim}, definitions, terminology and  lemmas are presented.
The main results and corollaries, and their proofs are given in Section~\ref{sec:main}.
Conjugation quandles are discussed in Sections~\ref{sec:conj} and \ref{sec:VC}.
Relations among various epimorphisms used in the proofs are examined in Secition~\ref{sec:rel}.

\section{Preliminaries}\label{sec:prelim}

In this section we briefly review some definitions and examples. 
More details can be found, for example, in \cite{CKS}. 

A {\it quandle} $X$ is a set with a binary operation $(a, b) \mapsto a * b$
satisfying the following conditions.
\begin{eqnarray}
\mbox{\rm (Idempotency) } & &  \mbox{\rm  For any $a \in X$,
$a* a =a$.} \label{axiom1} \\
\mbox{\rm (Right invertibility)}& & \mbox{\rm For any $b,c \in X$, there is a unique $a \in X$ such that 
$ a*b=c$.} \label{axiom2} \\
\mbox{\rm (Right self-distributivity)} & & 
\mbox{\rm For any $a,b,c \in X$, we have
$ (a*b)*c=(a*c)*(b*c). $} \label{axiom3} 
\end{eqnarray}
 A {\it quandle homomorphism} between two quandles $X, Y$ is
 a map $f: X \rightarrow Y$ such that $f(x*_X y)=f(x) *_Y f(y) $, where
 $*_X$ and $*_Y$ 
 denote 
 the quandle operations of $X$ and $Y$, respectively.
 A {\it quandle isomorphism} is a bijective quandle homomorphism, and 
 two quandles are {\it isomorphic} if there is a quandle isomorphism 
 between them.
 A quandle epimorphism $f: X \rightarrow Y$ is a {\it covering}~\cite{Eis3}
 if $f(x)=f(y)$ implies $a*x=a*y$ for all $a, x, y \in X$. 
   
  Let $X$ be a quandle.
  The {\it right translation}  ${R}_a:X \rightarrow  X$, by $a \in X$, is defined
by ${ R}_a(x) = x*a$ for $x \in X$. 
Then ${ R}_a$ is an automorphism of $X$ by Axioms (2) and (3). 
The subgroup of ${\rm Sym}(X)$ generated by the permutations ${ R}_a$, $a \in X$, is 
called the {\it {\rm inn}er automorphism group} of $X$,  and is 
denoted by ${\rm Inn}(X)$. 
The map ${\rm inn}: X \rightarrow {\rm inn}(X) \subset {\rm Inn}(X)$
defined by ${\rm inn}(x)=R_x$ is called the {\it inner representation}. 
An inner representation is a covering.

A quandle is {\it connected} if ${\rm Inn}(X)$ acts transitively on $X$.
A quandle is {\it faithful} if the mapping ${\rm inn}: X \rightarrow  {\rm Inn}(X)$ is an injection.

As in Joyce \cite{Joyce} a  quandle is defined  by  
a pair $(G, f)$ where $G$ is a  group,  $f \in {\rm Aut}(G)$,
and the quandle operation is defined by 
$x*y=f(xy^{-1}) y ,$ $ x,y \in G$.  We call such a quandle a {\it generalized Alexander quandle} and denote it by  ${\rm GAlex}(G,f)$.
If $G$ is abelian, such a quandle is  known as an {\it Alexander quandle}. 

In this paper all quandles are assumed to be finite.

\begin{figure}[htb]
    \begin{center}
   \includegraphics[width=3in]{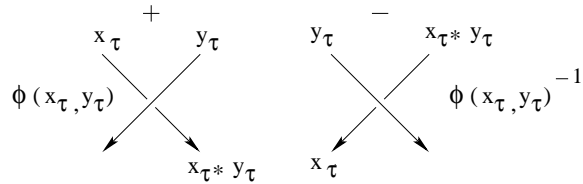}\\
    \caption{Colored crossings and cocycle weights }\label{coloredXing}
    \end{center}
\end{figure}

Let $D$ be a diagram of a knot $K$, and ${\cal A}(D)$ be the set of arcs of $D$.
A {\it  coloring}  of a knot diagram $D$ by a quandle $X$
is a map $C: {\cal A}(D) \rightarrow X$  satisfying the condition depicted in Figure~\ref{coloredXing}
at every 
 positive (left) and negative (right) crossing $\tau$,
respectively.  The pair $(x_\tau, y_\tau)$ of colors assigned to a pair of nearby arcs of a crossing $\tau$
is called the {\it source} colors, and the third arc is required to receive the color $x_\tau * y_\tau$.

In this paper we denote by $A$ a finite multiplicative abelian group whose identity element is denoted by $1$.
 A function $\phi: X \times X \rightarrow A$ for an abelian group $A$  is called
    a \emph{quandle $2$-cocycle}  \cite{CJKLS} if it satisfies $$ \phi (x, y)
    \phi(x,z)^{-1}  \phi(x*y, z)  \phi(x*z, y*z)^{-1}=1$$ 
    for any $x,y,z \in X$ and 
    $\phi(x,x)=1$ for any $x\in X$. For a quandle $2$-cocycle $\phi$, $E=X \times A$ becomes a quandle by setting 
    \[
    (x, a) * (y, b)=(x*y, a\, \phi(x,y))
    \]
    for $x, y \in X$, $a,b \in A$, denoted by
    $E(X, A, \phi)$ or simply $E(X, A)$, and it is called an \emph{abelian  extension} of $X$ by $A$. 
   Let $\pi: E (X, A) = X \times A \rightarrow X$ be the projection to the first factor.
    We also say that a quandle epimorphism $f: Y \rightarrow Z$ is an {\it abelian extension} 
    if there exists a isomorphisms $\nu:  E(X, A)  \rightarrow Y$ and $h: X \rightarrow Z$ such that $ h \pi= f  \nu$. 
    An abelian extension is a covering.
See  \cite{CENS} for more information on abelian
    extensions of quandles and \cite{CJKLS,CJKS,CJKS2} for more on quandle
    cohomology.

Let $X$ be a quandle, and $\phi$ be a $2$-cocycle with  coefficient group $A$,
a finite abelian group. Let $D$ be a diagram of a knot $K$.
The $2$-cocycle (or cocycle, for short) invariant is an element of the group ring $\Z [A]$ 
defined by $\Phi_{(X,A,\phi)} (D) =\Phi_{\phi} (D)= \sum_{C} \prod_{\tau} \phi(x_\tau, y_\tau)^{\epsilon(\tau)}$, where
the product ranges over all crossings $\tau$, the sum ranges over all colorings of a 
given knot diagram,
$(x_\tau, y_\tau)$ are source colors at the crossing $\tau$, and $\epsilon(\tau)$ 
is the sign of $\tau$ as specified in Figure~\ref{coloredXing}.
For a given coloring $C$, the element $\prod_{\tau} \phi(x_\tau, y_\tau)^{\epsilon(\tau)} \in A$ 
is denoted by $B_\phi (D, C) $.
For an abelian group $A$, the cocycle invariant takes the form $\sum_{a \in A} n_a a$ where $n_a \in \Z$,
and it is 
{\it constant} if $n_a=0$ when $a$ is not the identity of $A$.
It is known \cite{CJKLS} that $\Phi_\phi (D)$ is independent of choice of  diagram $D$ for a knot $K$. 
Thus we 
have the {\it 2-cocycle invariant} $\Phi_{(X,A, \phi )}(K) = \Phi_{ \phi }(K) = \Phi_\phi(D)$.

A  $1$-tangle is a properly embedded arc in a $3$-ball, and the equivalence of
$1$-tangles is defined by ambient  isotopies of the $3$-ball fixing the
boundary (cf.~\cite{Conway}).  A diagram of a $1$-tangle  is defined in a
manner similar to a knot diagram, from a regular projection to a disk by
specifying  crossing information, see Figure~\ref{tangles}(A).  An orientation
of a $1$-tangle is specified by an arrow on a diagram as depicted.  A knot
diagram is obtained from a $1$-tangle diagram by closing the end points by a
trivial arc outside of a disk. This procedure is called the {\it closure} of a
$1$-tangle.  If a $1$-tangle is oriented, then the closure inherits the
orientation.
Two  diagrams of the same $1$-tangle are related by Reidemeister moves.
There is a bijection between knots and $1$-tangles for classical knots, and invariants of 1-tangles give rise to
invariants of knots, see \cite{Eis3}, for example.

\begin{figure}[htb]
    \begin{center}
   \includegraphics[width=2.7in]{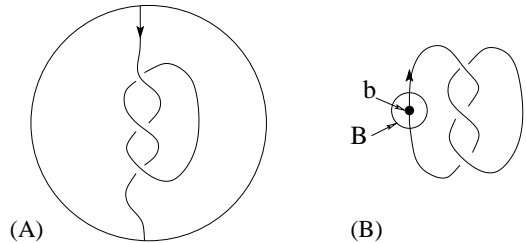}\\
    \caption{ $1$-tangles }\label{tangles}
    \end{center}
\end{figure}

A $1$-tangle is obtained from a knot $K$ as follows.  Choose   a base point $b
\in K$  and a small closed ball neighborhood  $B$ of $b$ in the $3$-sphere $\Sym^3$
such that  $(B, K\cap B)$ is a trivial ball-arc pair (so that $K \cap B$ is
unknotted in $B$, see Figure~\ref{tangles}(B)).  Then  $(\Sym^3 \setminus {\rm
Int}(B), K \cap (\Sym^3 \setminus {\rm Int}(B)))$ is a $1$-tangle called the
$1$-tangle associated with $K$.  The resulting $1$-tangle does not depend on
the choice of a base point.  
 If  a knot is oriented, then the corresponding $1$-tangle inherits the
orientation. 

A quandle coloring of an oriented $1$-tangle diagram is defined in a  manner
similar to  those for knots.  We do not require that the end points receive the
same color for a quandle coloring of $1$-tangle diagrams.
As in \cite{CSV} we say that a quandle $X$ is {\it end monochromatic} for a tangle diagram $T$ if 
any coloring of $T$ by $X$ assigns the same color on the two end arcs.
We use the same notations  $\Phi_{\phi} (T) = \sum_{\cal C} \prod_{\tau} \phi(x_\tau, y_\tau)^{\epsilon(\tau)}$
and $B_\phi (T, C) $ for tangle diagrams $T$. This $\Phi_{\phi} (T)$ is again independent of choice of a diagram,
and an invariant of tangles.
Figures~\ref{coloredXing}, \ref{tangles}, \ref{endmono} are taken from \cite{CSV}. 

\begin{figure}[htb]
    \begin{center}
   \includegraphics[width=1.1in]{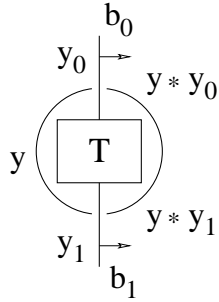}\\
    \caption{ Colorings of a  tangle }\label{endmono}
    \end{center}
\end{figure}

We recall  the following two lemmas.
 
\begin{lemma}[Eisermann~\cite{Eis3}, Theorem 30] \label{lem:cover}
Let $f: Y \rightarrow X$ be a covering, and 
$C_X: {\cal A} (T) \rightarrow X$ be a coloring of a $1$-tangle $T$ by $X$. 
Let $b_0, b_1$ be the top and bottom arcs as depicted in Figure~\ref{endmono}. 
Then  for any $y \in Y $ such that $f(y)=C_X(b_0)$,  
there exists a unique coloring $C_Y: {\cal A} (T) \rightarrow Y$ such that 
$f  C_Y=C_X$ and $C_Y(b_0)=y$.
\end{lemma}

\begin{lemma}[\cite{CENS,CSV} ] \label{lem:const}
Let  
$E=E(X, A, \phi)$ be an abelian extension for a $2$-cocycle $\phi$. 
Then   $E$ is end monochromatic  $T$  if and only if 
 $\Phi_{(X, A, \phi)}(K)$ is constant, where $T$ is a $1$-tangle for a knot $K$.
\end{lemma}

\begin{lemma}\label{lem:end}
Let  $C:  {\cal A} (T) \rightarrow Y$ be a coloring of a classical $1$-tangle diagram $T$ by a quandle $Y$. 
For the top and bottom arcs $b_0$ and $b_1$ of $T$, respectively, 
let $y_0=C(b_0)$ and $y_1=C(b_1)$.
Then  ${\rm inn}(y_0)=R_{y_0}=R_{y_1}={\rm inn}(y_1)$. 
\end{lemma}

\begin{proof}
The proof in \cite{CSV}, based on corresponding statements in \cite{Nos,Jozef,Jozef2004} on faithful quandles, applies in this situation. The idea of proof is seen in Figure~\ref{endmono}. 
The large circle behind the tangle $T$ in the figure can be pulled out of $T$ if $T$ corresponds to a classical knot. 
Colorings of tangle diagrams can be defined when a tangle has more than one component, 
 in a similar manner.
Hence any  color $y$ at the left 
of the large circle 
should extend to the color to the right, so that for the colors $y_0$ and $y_1$ for 
the top and bottom arcs $b_0$ and $b_1$ must satisfy $y*y_0=y*y_1$, hence we obtain 
$R_{y_0}=R_{y_1}$. 
\end{proof}

\begin{figure}[htb]
    \begin{center}
   \includegraphics[width=3.5in]{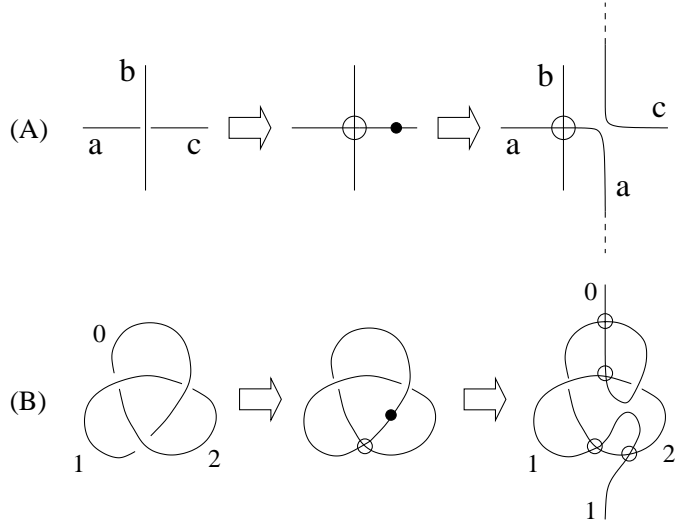}\\
    \caption{ Long virtual knots }\label{virtualize}
    \end{center}
\end{figure}

\begin{remark}
{\rm
To see that the condition of being a classical knot is essential, 
we observe examples of {\it long virtual knots} (i.e., virtual $1$-tangles)
 for which the conclusion of Lemma~\ref{lem:end}
 does not hold.
  See \cite{Man}, for example, for long virtual knots.
  Colorings of oriented virtual (and long virtual) knots are defined in a manner similar to classical knots, with colors unchanged for each transverse arc at 
  each virtual crossing
  \cite{Kauff}. 
Suppose a (virtual or classical) knot diagram $D$  is colored non-trivially by a faithful quandle $X$, 
and let $\tau$ be a non-trivially colored crossing (that is, distinct colors appear at $\tau$). 
Since $X$ is faithful, all three colors at $\tau$ are distinct (see Figure~\ref{virtualize}(A) left). 
Make $\tau$ a virtual crossing (called {\it virtualization of a crossing}) as in the middle of (A),
and cut and prolong one of the under-arcs up and down, to make a long virtual knot (a virtual $1$-tangle),
as in (A) right. A cut point is indicated by a dot in (A) middle. Prolonged arcs cross other arcs by virtual crossings. 
Then the top and bottom arcs are colored by elements $a, c \in X$ such that $R_a \neq R_c$. 
Thus the resulting long virtual knot is not classical.
An example of this construction is illustrated in Figure~\ref{virtualize}(B) with a Fox tri-coloring.
The quandle used here is $\Z_3$ with $a*b=2b-a$ mod 3, a well known {\it dihedral} quandle of order 3, which is faithful. 
}
\end{remark}

\section{
Cocycle invariants and  the images of inner representations
}\label{sec:main}

In this section we 
present relations between certain algebraic properties  of quandles and values of the quandle cocycle invariant.
We do not assume connectivity of quandles in this section unless otherwise specified.

\begin{theorem}\label{thm:main}
Let $E=E(X,A,\phi)$ be an abelian extension of a quandle $X$.
If $E \cong {\rm inn}(Y)$ for some quandle $Y$,
then $\Phi_{(X,A,\phi)}(K)$ is constant for all classical knots $K$.
\end{theorem}

\begin{proof}
Let $T$ be a $1$-tangle of $K$,  $b_0, b_1$ be the top and bottom arcs of $T$, respectively.
Let $C$ be a coloring of a diagram of $K$ by  $X$, and use the same notation  $C: {\cal A}(T) \rightarrow X$ 
for a corresponding coloring of $T$ such that $C(b_0)=C(b_1)=x \in X$.
Then $C$ extends to a coloring $C_E: {\cal A}(T) \rightarrow E$ 
by Lemma~\ref{lem:cover}. 

Recall that the inner representation is a covering. 
Without loss of generality, we assume that ${\rm inn}(Y)=E$.
By assumption and Lemma~\ref{lem:cover}, $C_E$ extends to a coloring 
$C_Y$. Since $E$ is an abelian extension of $X$, 
Lemma~\ref{lem:const}  and Lemma~\ref{lem:end} imply 
that the cocycle invariant is constant.
\end{proof}

To apply the theorem to some Rig quandles, we observe the following.

\begin{lemma}\label{lem:order2}
If ${\rm inn} : Y \rightarrow {\rm inn}(Y)= X$, for connected quandles $X$ and $Y$,  satisfies $|Y|/|X|=2$, then 
${\rm inn}$  is an abelian extension.
\end{lemma}

\begin{proof}
For ${\rm inn}: Y \rightarrow {\rm inn}(Y)=X$, where  $X$ and $Y$ are connected quandles, it is proved in \cite{AG} that 
there is a quandle isomorphism 
$\nu :  X \times S  \rightarrow Y$ 
for a set $S$, 
such that 
$\pi={\rm inn} ( \nu) $ 
for the projection $\pi: X \times S \rightarrow X$. 
The 
 quandle operation on $X \times S$ is defined by 
$$ (x, s)*(y, t) = (x*y, \beta_{x,y} (s) )\quad {\rm for } \quad (x, s), (y, t) \in X \times S, $$
for  some  $\beta: X^2 \rightarrow {\rm Sym}(S)$.
The proof of Theorem 7.1 of \cite{CSV} shows  that if 
the cardinality of $S$ is 2, then we can assume $S=\Z_2$ and $\beta_{x, y}(a) = a\, \phi (x,y)$
where $\phi$ is a $2$-cocycle with coefficient group $A=\Z_2$.  Hence ${\rm inn}$ is an abelian extension. 
\end{proof}

We say that an epimorphism  $f: Y \rightarrow X$  is {\it of  index $k$} if $|Y|/|X|=k$.
Note that if $Y$ is connected, then each fiber $f^{-1}(x)$, $x \in X$, has cardinality $k$.

\begin{corollary}\label{cor:rig}
The following Rig quandles have non-trivial second cohomology groups with the coefficient group $A=\Z_2$, yet 
give rise to   constant quandle $2$-cocycle invariants for any classical knot with  the corresponding non-trivial $2$-cocycles: 
$$
\begin{array}{llllllll}
  Q(6,1),&  Q(10,1), &  
Q(12, 5), & Q(12, 6), &  Q(12, 7), & Q(12, 8), & 
 Q(16,4), &
  Q(16,5), \\
  Q(16,6), & 
 Q(18, 1), &  
   Q(18, 2), &  
    Q(18, 8), &  
   Q(18, 9), &  Q(18, 10), & 
      Q(24, 3), &
      Q(24, 4), \\
    Q(24, 13), &  
 Q(24, 22),  &  
 Q(30, 2), &  Q(30, 7), &   
  Q(30, 8), & 
Q(40, 8), & 
 Q(40, 9), &    Q(40, 10), \\  
Q(42, 1), &  Q(42, 3), &  Q(42, 4), &  Q(42, 7), &  Q(42, 8). & & & 
\end{array}
$$
\end{corollary}

\begin{proof}
Computer calculations show the following quandle sequences of Rig quandles:
$$ 
\begin{array}{rllll}
 Q(24,1)& \stackrel{{\rm inn}}{\longrightarrow} & Q(12,1) & \stackrel{{\rm inn}}{\longrightarrow} & Q(6,1),  \\
  Q(40,2) &\stackrel{{\rm inn}}{\longrightarrow} &  Q(20, 3)  & \stackrel{{\rm inn}}{\longrightarrow}&  Q(10,1) . \\
\end{array}
$$
For example, the first sequence shows that $Q(12, 1)$ is an abelian extension of $Q(6,1)$ by a cocycle $\phi$, 
and is 
the image of $Q(24,1)$ by ${\rm inn}$. Then  Lemma~\ref{lem:order2} and Theorem~\ref{thm:main}
imply that $\Phi_\phi (K)$ is constant for all classical knots $K$.
For the other quandles, the result follows in a similar fashion 
from  sequences of inner representations 
in Tables~\ref{table:4term} and~\ref{table:short}.
See \cite{Edwin2} for multiplication tables of quandles in Tables \ref{table:4term} and \ref{table:short}.
\end{proof}

\begin{table}[h]
$$
\begin{array}{lllllll}
R(192,2) & {\rightarrow } & R(48,3) & {\rightarrow } & Q(24,3) & {\rightarrow } & Q(12,6)  \\ 
R(192,3) & {\rightarrow } & R(48,4) & {\rightarrow } & Q(24,4) & {\rightarrow } & Q(12,5)   
\end{array}
$$
\caption{Four-term sequences of inner representations for small connected quandles}\label{table:4term}
\end{table}

\begin{table}[hbt]
$$
\begin{array}{lllllllllll}
R(64,1) & {\rightarrow } & Q(32,6) & {\rightarrow } & Q(16,4) &R(64,2) & {\rightarrow } & Q(32,7) & {\rightarrow } & Q(16,5)  \\ 
R(64,3) & {\rightarrow } & Q(32,8) & {\rightarrow } & Q(16,6) &R(64,4) & {\rightarrow } & Q(32,5) & {\rightarrow } & Q(16,4)  \\ 
R(64,5) & {\rightarrow } & Q(32,6) & {\rightarrow } & Q(16,4) &R(64,6) & {\rightarrow } & Q(32,5) & {\rightarrow } & Q(16,4)  \\ 
R(72,1) & {\rightarrow } & Q(36,21) & {\rightarrow } & Q(18,10) &R(72,2) & {\rightarrow } & Q(36,17) & {\rightarrow } & Q(18,8)  \\ 
R(72,3) & {\rightarrow } & Q(36,20) & {\rightarrow } & Q(18,9) &R(72,4) & {\rightarrow } & Q(36,4) & {\rightarrow } & Q(18,2)  \\ 
R(72,5) & {\rightarrow } & Q(36,1) & {\rightarrow } & Q(18,1) &R(96,1) & {\rightarrow } & Q(24,4) & {\rightarrow } & Q(12,5)  \\ 
R(96,2) & {\rightarrow } & Q(24,3) & {\rightarrow } & Q(12,6) &R(96,3) & {\rightarrow } & R(48,1) & {\rightarrow } & Q(24,22)  \\ 
R(96,4) & {\rightarrow } & Q(24,6) & {\rightarrow } & Q(12,9) &R(96,5) & {\rightarrow } & Q(24,6) & {\rightarrow } & Q(12,9)  \\ 
R(96,6) & {\rightarrow } & Q(24,6) & {\rightarrow } & Q(12,9) &R(96,7) & {\rightarrow } & Q(24,5) & {\rightarrow } & Q(12,8)  \\ 
R(96,8) & {\rightarrow } & Q(24,5) & {\rightarrow } & Q(12,8) &R(96,9) & {\rightarrow } & R(48,2) & {\rightarrow } & Q(24,22)  \\ 
R(120,1) & {\rightarrow } & Q(20,3) & {\rightarrow } & Q(10,1) &R(120,2) & {\rightarrow } & R(60,1) & {\rightarrow } & Q(30,8)  \\ 
R(120,3) & {\rightarrow } & R(60,2) & {\rightarrow } & Q(30,7) &R(120,4) & {\rightarrow } & R(60,3) & {\rightarrow } & Q(30,2)  \\ 
R(120,5) & {\rightarrow } & Q(30,1) & {\rightarrow } & Q(15,2) &R(160,1) & {\rightarrow } & R(80,1) & {\rightarrow } & Q(40,10)  \\ 
R(160,2) & {\rightarrow } & Q(40,20) & {\rightarrow } & Q(20,5) &R(160,3) & {\rightarrow } & Q(40,19) & {\rightarrow } & Q(20,6)  \\ 
R(160,4) & {\rightarrow } & R(80,2) & {\rightarrow } & Q(40,9) &R(168,1) & {\rightarrow } & R(84,1) & {\rightarrow } & Q(42,8)  \\ 
R(168,2) & {\rightarrow } & R(84,2) & {\rightarrow } & Q(42,3) &R(168,3) & {\rightarrow } & R(84,3) & {\rightarrow } & Q(42,7)  \\ 
R(168,4) & {\rightarrow } & R(84,4) & {\rightarrow } & Q(42,4) &R(168,5) & {\rightarrow } & R(84,5) & {\rightarrow } & Q(42,1)  \\ 
R(192,1) & {\rightarrow } & Q(24,14) & {\rightarrow } & Q(12,7) &R(192,4) & {\rightarrow } & R(48,5) & {\rightarrow } & Q(24,13)  \\ 
R(216,1) & {\rightarrow } & R(72,6) & {\rightarrow } & Q(24,21) &R(216,2) & {\rightarrow } & Q(36,17) & {\rightarrow } & Q(18,8)  \\ 
R(216,3) & {\rightarrow } & Q(36,1) & {\rightarrow } & Q(18,1) &
& 
\end{array}
$$
\caption{Three-term sequences of inner representations for small connected quandles}\label{table:short}
\end{table}

We remark that the above list contains all Rig quandles of order less than or equal to 16 that 
were conjectured in \cite{CSV} 
 to have  constant quandle $2$-cocycle invariants for any classical knot with  non-trivial $2$-cocycles 
except 
$Q(15,2)$, and $Q(15,7)$, 
which will be shown to have constant invariant in Section~\ref{sec:VC}.
Those in the above list of order larger than 16 do not  appear in the  conjectured  list in \cite{CSV}.

Tables~\ref{table:4term} and~\ref{table:short} contain sequences of connected quandles where all arrows represent inner representations. 
The quandle on the left of each sequence is a  generalized Alexander quandle, but others in the sequence may or may not be generalized Alexander quandles. The  right-most quandle in each sequence is faithful, so the sequences cannot be extended non-trivially to the right with inner representations.
The notation $R(n, j)$ is used to indicate a quandle of order $n$ when $n >  47$ and hence not a Rig quandle. The  index
$j$ is simply to distinguish non-isomorphic quandles.

\begin{remark}
{\rm
Terminating sequences of ${\rm inn}$, 
$$X=X_0 \stackrel{{\rm inn}}{\longrightarrow} X_1={\rm inn}(X_0) \stackrel{{\rm inn}}{\longrightarrow} \cdots
 \stackrel{{\rm inn}}{\longrightarrow} X_n={\rm inn}(X_{n-1}) , $$
 are discussed in \cite{AG},  where $X_n$ is faithful and $X_j$ are not, for $j=1, \ldots, n-1$. 
 For the 790 Rig quandles of order less than 48, 
there are 66 non-faithful quandles $X$, and all but two have  faithful images ${\rm inn}(X)$.
The two exceptions are the above first two Rig quandles $Q(24, 1)$ and $Q(40, 2)$
in Corollary~\ref{cor:rig}. 

On the other hand, Tables~\ref{table:4term} and~\ref{table:short} include many 
quandles $X$
with  ${\rm inn}(X)$ being non-faithful Rig quandles.
}
\end{remark}

 Let $X=Q(12, 5)$ or $Q(12,6)$.  Then the second quandle cohomology group $H^2_Q(X,\Z_4)$ is known \cite{rig}
 to be isomorphic to $\Z_4$. See \cite{CJKLS,CKS}, for example, for details on quandle cohomology.
Let $\psi: X \times X \rightarrow \Z_4$ be a $2$-cocycle which represents  
 a generator of $H^2_Q(X,\Z_4)\cong \Z_4$. 
 Let $u$ denote a multiplicative generator of $A=\Z_4$. 
The cocycle invariants $\Phi_\psi(K)$ for $X=Q(12, 5)$ or $Q(12,6)$ with respect to  $\psi$,  computed for some knots in the table  in \cite{rig} up to 9 crossing knots,  
contain non-constant values, while for $A=\Z_2$ the invariant is  constant by Corollary~\ref{cor:rig}.
This is explained by the following.

\begin{theorem}
Let $X$ be a quandle and $n, m, d>1$  be positive integers such that $n=md$. 
Let $\psi$ be a $2$-cocycle of $X$ with values in $\Z_n$, and 
 $\Phi_\psi(K)= \sum_{j=0}^{n-1} a_j(K)\,  u^j $ be the cocycle invariant of a knot $K$ with respect to $\psi$. 
 
Let $E=E(X, \Z_m, \phi) \stackrel{\alpha}{\rightarrow} X$ be the abelian extension 
corresponding to $\phi=\psi^d$,
and suppose that there is a sequence of quandles 
$Y \stackrel{\rm inn}{\longrightarrow} E \stackrel{\alpha}{\rightarrow} X$. 
Then $a_k (K)=0$ for all $k$ that are not divisible by  $m$,  for any classical knot $K$. 
\end{theorem}

\begin{proof}
For each coloring $C: {\cal A}(T) \rightarrow X$,
the weight $\phi(x_\tau, y_\tau )^{\epsilon (\tau) }$  at a crossing $\tau$ 
for $\phi$
is equal to $\psi(x_\tau, y_\tau )^{d \epsilon (\tau) }$. Hence if 
$B_\psi(T, C)= u^j$, then we have $B_\phi(T, C)= u^{dj}$. 
Then 
\begin{eqnarray*}
\Phi_\phi(K) &=&  \sum_{j=0}^{n-1} a_j(K)\, u^{dj} \\
&=& [ \, \sum_{k=0}^{d-1} a_{mk} (K)  \, ] \, u^0 +
[ \, \sum_{k=0}^{d-1} a_{1+mk} (K)  \, ] \, u^d + \cdots \\
& & + \  [ \, \sum_{k=0}^{d-1} a_{\ell + mk} (K)  \, ] \, u^{d \ell}  + \cdots + \,
[ \, \sum_{k=0}^{d-1} a_{(d-1)d + mk} (K)  \, ] \, u^{ d (d-1) } .
\end{eqnarray*}
By Theorem~\ref{thm:main}, this is equal to a constant,  $\sum_{k=0}^{d-1} \, [ \, a_{mk} (K)  \, ] \, u^0$. 
Since $a_j(K)$ are non-negative integers, 
the result follows.
\end{proof}

The following solves a conjecture stated in \cite{CSV}.

\begin{corollary}\label{cor:Z4}
Let $X=Q(12, 5)$ or $Q(12,6)$, and $\psi: X \times X \rightarrow A=\Z_4$ be a $2$-cocycle which represents  
 a generator of $H^2_Q(X,\Z_4)\cong \Z_4$. Let $\Phi_\psi(K)= \sum_{j=0}^3 a_j(K)\, u^j \in \Z[A]$ be the cocycle invariant.
Then $a_1(K)=a_3(K)=0$ for any classical knot $K$.
\end{corollary}

This situation is also found for $X=Q(18,1)$ or $Q(18,8)$, where $H^2_Q(X,\Z_6)\cong \Z_6$. 
Let $u$ be a multiplicative generator of $A=\Z_6$. Then the invariant values are restricted to the following form.

\begin{corollary}\label{cor:Z6}
Let $X=Q(18, 1)$ or $Q(18,8)$, and $\psi: X \times X \rightarrow \Z_6$ be a $2$-cocycle which represents  
 a generator of $H^2_Q(X,\Z_6)\cong \Z_6$. Let $\Phi_\psi(K)= \sum_{j=0}^5 a_j(K)\, u^j \in \Z[A]$ be the cocycle invariant.
Then $a_k(K)=0$ for $k=1,3,5$ for any classical knot $K$.
\end{corollary}
 
 The cocycle invariant for connected quandles of order 18 are computed 
 in \cite{rig} for up to 7 crossing knots at the time of writing.  The invariant values for $Q(18,8)$ do contain non-constant values.
 For $Q(18,1)$, the invariant is constant, and we do not know whether this is an artifact of limited number of knots or it is constant 
 for all classical knots.

\bigskip

Theorem~\ref{thm:main} can be applied contrapositively:
if $E \rightarrow X$ is an abelian extension with a $2$-cocycle $\phi$ such that 
$\Phi_\phi(K)$ is not  constant for a classical knot $K$, then 
there is no finite quandle $Y$ such that ${\rm inn}(Y)=E$ for the inner representation ${\rm inn}$.
Thus we obtain the following from  \cite{Edwin} and \cite{rig}. 
The list contains abelian extension $E(X,\Z_2, \phi)$  such that $\Phi_\phi(K)$ is non-constant for some knot $K$ in the table.
The list is from the information available at the time of writing.

\begin{corollary}\label{cor:no-inn}
For the following Rig quandles $E$, there is no
 finite quandle $Y$ such that ${\rm inn}(Y)=E$ for the inner representation ${\rm inn}$:
$$
Q(8,1),  \ 
Q(24,1), \ 
Q(24, 7), \ 
Q(32, 1). \ 
$$
\end{corollary}

The arguments in Corollary~\ref{cor:Z4} can be applied contrapositively to obtain the following as well.

\begin{corollary}\label{cor:noY}
Let $E$ be one of the Rig quandles $ Q(12,2)$, 
 $Q(36,1)$ or 
$Q(32,9)$. Then 
 there is no finite quandle $Y$ such that ${\rm inn}(Y)=E$ for the inner representation ${\rm inn}$.
 \end{corollary}

\begin{proof}
Each of the Rig quandles $Q(6,2)$, $Q(16,1)$, $Q(16, 7)$ has the second cohomology group isomorphic to $\Z_4$.
It is computed in \cite{Edwin} that the  index 2 abelian extensions corresponding to $\psi^2$ for a 2-cocycle $\psi$ representing a generator of 
the second cohomology group are, respectively,  $Q(12,2)$, 
$ Q(36,1)$, and 
$Q(32,9)$.

Let $E$ be one of $Q(12,2)$, 
$ Q(36,1)$, and 
$Q(32,9)$. If there is a finite quandle $Y$ such that ${\rm inn}(Y)=E$, then by the proof of 
Corollary~\ref{cor:Z4}, the cocycle invariant  $\Phi_\psi(K)= \sum_{j=0}^{3} a_j(K)\, u^{j} $
satisfies $a_1(K)=a_3(K)=0$ for all classical knots $K$.
However, this is not the case 
from \cite{rig}, and the result follows.
\end{proof}

\begin{remark}
{\rm
Similar   arguments apply  to 
other Rig quandles for $Q$ with cyclic second cohomology groups,  
but the corresponding index 2 abelian extensions go out of bounds 
of Rig quandles, and we are not able to specify the quandles with this property.
However, we can conclude that 
for   each  Rig quandle $Q$  in the list given below, 
there is a non-trivial abelian extension
$\alpha: X \rightarrow Q$  of index $2$  with the property that 
there is no finite quandle $Y$ such that ${\rm inn}(Y)=X$ for the inner representation ${\rm inn}$:
$$
  Q(18,3), \
Q(18,6), \  Q(18,7), \
Q(12, 10), \ 
Q(12, 3),\ 
Q(18,5). 
$$
}
\end{remark}

\section{$\Phi_{(X,A,\phi)}(K)$ is constant if $E(X,A,\phi)$ 
is a conjugation quandle} \label{sec:conj}

Let $G$ be a finite group. For $a,b \in G$ we write $a^b = b^{-1}ab$ and denote the conjugacy class of $G$ containing $x$ by $x^G$.  
The conjugacy class $x^G$ under conjugation, $a*b = a^b$, is a quandle.  
Here we call such a quandle a {\it conjugation quandle}.  
We note that such a quandle need not be connected. 
In general, a subquandle of a group $G$ under conjugation need not be a conjugacy class.
But  it is easy to see that if
$X$ is a subset of a group $G$ closed under conjugation and if $X$ under conjugation is
a connected quandle then $X$ is a conjugacy class of the group $\langle X \rangle$ generated by $X.$
Note that we are only interested in this paper in the case where both $X$ and $G$ are finite.

We show in this section that a connected quandle $E$ satisfies the condition in Theorem~\ref{thm:main} if and only if $E$ is a conjugation quandle.

To simplify the proof of the next theorem we will need the following 
lemma.

\begin{lemma} \label{lem:triangle} Suppose $A$, $B$, and $C$ are quandles, $f$ and $g$ are quandle epimorphisms and $h$ is a well-defined bijection such that $g = hf$, 
then $h$ is a quandle isomorphism.
\end{lemma}

\begin{proof} Let $a_1, a_2 \in A$. Since $f$ is an epimorphism there exist $c_1,c_2 \in B$ such that $f(c_1) = a_1$ and $f(c_2) = a_2$. Then using $g = hf$ we have
 $h(a_1*a_2) = h(f(c_1)*f(c_2)) =
h(f(c_1*c_2)) = g(c_1*c_2) = g(c_1)*g(c_2) = hf(c_1)*hf(c_2) = h(a_1)*h(a_2)$.
\end{proof}

\begin{theorem}\label{thm:conjugation_criterion} If $X$ is a conjugation quandle, then $X \cong {\rm inn}(Y)$ for some quandle $Y$.
\end{theorem}

\begin{proof} Let $X = x^G$ be a conjugation quandle for some finite group $G$ and element $x \in G$.
Define $f \in {\rm Aut}(G)$ by $f(a) = x^{-1}ax$, $a \in G$ and let $Y$ be the generalized Alexander quandle ${\rm GAlex}(G,f)$. Define $p: Y \rightarrow X$ by $p(g) = x^g$. An easy calculation shows that $p$ is a homomorphism. The homomorphism $p$ is clearly surjective as is the homomorphism ${\rm inn} : Y \rightarrow {\rm inn}(Y)$. Now we claim that the mapping 
$\varphi: {\rm inn}(Y) \rightarrow X$ given by $\varphi(R_g) = x^g$  is a bijection and hence by Lemma~\ref{lem:triangle} it is a quandle isomorphism. The see this we note that $R_g = R_h$ holds if and only if  $a*g = a*h$ for all $a \in G$. That is, if and only if $x^{-1} a g^{-1} x g = x^{-1} a h^{-1} x h$ which by cancellation is equivalent to $x^g = x^h$. This completes the proof.
\end{proof}

\begin{remark} 
{\rm 
We note that if $Y$ is connected then $X = {\rm inn}(Y)$ is a conjugacy class in ${\rm Inn}(Y)$. So for connected $X$ the converse of Theorem~\ref{thm:conjugation_criterion} holds. 
In fact, by using Lemma 25 in \cite{Eis3}  if $X$ is  a connected conjugation quandle one may even find a connected quandle $Y$ such that $X \cong {\rm inn}(Y)$. 
However, we do not assume connectivity for the results of this section.
}
\end{remark}

The following is immediate from Theorems~\ref{thm:main} and \ref{thm:conjugation_criterion}.

\begin{theorem} \label{ext_constant} If $\phi: Q \times Q \rightarrow A$ is a 2-cocycle of a quandle $Q$ with an abelian coefficient group $A$
such that the extension $X=E(Q,A,\phi)$ is a  conjugation quandle, then $\Phi_\phi(K)$ is constant for all classical knots $K$.
\end{theorem}

Theorem~\ref{thm:conjugation_criterion} and Corollaries~\ref{cor:no-inn}, \ref{cor:noY} 
imply the following as well.

\begin{corollary}\label{cor:non-faith}
The following Rig quandles  are not conjugation quandles:
$$ Q(8,1),  \  Q(12,2), \ Q(24,1), \ Q(24,7), \  Q(32,1), \ Q(32, 9),  \ Q(36, 1). $$
\end{corollary}

\section {A criterion for a finite connected quandle to be a conjugation quandle}\label{sec:VC}

We are grateful to Leandro Vendramin for suggesting the  criterion in this section
for a connected quandle to be a conjugation quandle and for pointing out the
relevance of Lemma 1.8 in \cite{HV}  to its proof.

The development here is from the paper \cite{HV} by Vendramin and Heckenberger as well as the paper \cite{GHV}.
For a (finite) quandle $X=(X,*)$, its {\it enveloping group} is the group 
$G_X$ with presentation 
    $$G_X =\langle x_i , i \in X \ |\  x_{i*j} = x_j^{-1}x_i x_j, i, j \in X \rangle$$ 
and with the natural mapping 
          $$\partial : X \rightarrow G_X,\  i \mapsto x_i.$$ 
This group has been defined and used several times previously,  
Joyce \cite{Joyce} denoted it by $ {\rm Adconj}(X)$, 
Fenn and Rourke called it the {\it associated group} and denoted it by ${\rm As}(X)$ \cite{FR}, 
and Eisermann \cite{Eis4}   called it the {\it adjoint group} and denoted it by ${\rm Adj}(X)$.
The group $\overline{G_X}$ obtained from
$G_X$ by adding the  relations ${x_i^{n_i}, i \in  X}$ where $n_i$ is the order of $R_i$ is called the {\it finite enveloping group}. That this group is finite was proved in \cite{GHV}, Lemma 2.17. As in \cite{HV} we write $$ \pi : G_X \rightarrow \overline{G_X}$$ for the canonical surjection.  Let 
  $$\rho : X \rightarrow  \overline{G_X}$$
be the natural mapping from $X$ to $\overline{G_X}$, that is $\rho = \pi \partial$. It is clear that $\rho$ is a quandle homomorphism to $\overline{G_X}$ considered as a quandle under conjugation. Then we have the following theorem.

\begin{theorem} \label{thm:VC} (Vendramin's Criterion \cite{Leandro2}) A connected quandle $X$ is a conjugation
quandle if and only if  $\rho : X \rightarrow  \overline{G_X}$ is an injection.
\end{theorem}

\begin{proof} It is clear that the mapping $\partial : X \rightarrow G_X, i \mapsto x_i$ is universal for quandle homomorphisms $f:X\rightarrow G$ where G is considered as a conjugation quandle. That is,
for any such $f$ there is a unique group homomorphism $g: G_X \rightarrow G$ such that $f = g\partial$.  It then follows that if $f:X\rightarrow G$ is an injection, that is, if $X$ is a conjugation
quandle then $\partial$ must also be an injection.  By Lemma 1.8 in \cite{HV}  $\partial$ is  injective implies that $\rho$ is injective. Since $X$ is connected the converse is clear.
\end{proof}

Clearly each faithful connected quandle is a conjugation quandle. Among the 790 connected quandles of order at most 47 there are precisely 66 non-faithful quandles.  
Using \GAP/Rig and Vendramin's Criterion we find that exactly   30 of the 66 non-faithful Rig quandes are conjugations quandles. 
These are found in Table~\ref{table:conj}.

\begin{table}[htb]
$$
\begin{array}{llllllll}
Q(12,1) & Q(20,3) & Q(24,3) & Q(24,4) & Q(24,5) & Q(24,6) \\
Q(24,14) & Q(24,16) & Q(24,17) & Q(30,1) & Q(30,16) & Q(32,5) \\ 
Q(32,6) & Q(32,7) & Q(32,8) & Q(36,1) & Q(36,4) & Q(36,17) \\ 
Q(36,20) & Q(36,21) & Q(36,28) & Q(36,30) & Q(40,12) & Q(40,13) \\ 
Q(40,17) & Q(40,18) & Q(40,19) & Q(40,20) & Q(42,12) & Q(42,21) \\ 

\end{array}
$$
\caption{Non-faithful conjugation Rig quandles}\label{table:conj}
\end{table}

These are listed in \cite{Edwin1} together with groups in which they are conjugation quandles. Furthermore 
each of these 30 quandles is an abelian extension of index 2 of the form ${\rm inn}: Q \rightarrow {\rm inn}(Q)$, as indicated in 
Table~\ref{table:innconj}. 
By  Theorem~\ref{ext_constant} the corresponding cocycle invariants of these extensions are constant for all knots. Note for example
that $Q(12,8)$ has three different such cocycle extensions. This is possible since 
$H^2_Q(X, \Z_2) \cong \Z_2 \times \Z_2 \times \Z_2$.

\begin{table}[h]
$$
\begin{array}{llllll}
{\rm inn}(Q(12,1)) & \cong & Q(6,1) & {\rm inn}(Q(20,3)) & \cong & Q(10,1) \\ 
{\rm inn}(Q(24,3)) & \cong & Q(12,6) & {\rm inn}(Q(24,4)) & \cong & Q(12,5) \\ 
{\rm inn}(Q(24,5)) & \cong & Q(12,8) & {\rm inn}(Q(24,6)) & \cong & Q(12,9) \\ 
{\rm inn}(Q(24,14)) & \cong & Q(12,7) & {\rm inn}(Q(24,16)) & \cong & Q(12,8) \\ 
{\rm inn}(Q(24,17)) & \cong & Q(12,8) & {\rm inn}(Q(30,1)) & \cong & Q(15,2) \\ 
{\rm inn}(Q(30,16)) & \cong & Q(15,7) & {\rm inn}(Q(32,5)) & \cong & Q(16,4) \\ 
{\rm inn}(Q(32,6)) & \cong & Q(16,4) & {\rm inn}(Q(32,7)) & \cong & Q(16,5) \\ 
{\rm inn}(Q(32,8)) & \cong & Q(16,6) & {\rm inn}(Q(36,1)) & \cong & Q(18,1) \\ 
{\rm inn}(Q(36,4)) & \cong & Q(18,2) & {\rm inn}(Q(36,17)) & \cong & Q(18,8) \\ 
{\rm inn}(Q(36,20)) & \cong & Q(18,9) & {\rm inn}(Q(36,21)) & \cong & Q(18,10) \\ 
{\rm inn}(Q(36,28)) & \cong & Q(18,12) & {\rm inn}(Q(36,30)) & \cong & Q(18,11) \\ 
{\rm inn}(Q(40,12)) & \cong & Q(20,9) & {\rm inn}(Q(40,13)) & \cong & Q(20,9) \\ 
{\rm inn}(Q(40,17)) & \cong & Q(20,10) & {\rm inn}(Q(40,18)) & \cong & Q(20,9) \\ 
{\rm inn}(Q(40,19)) & \cong & Q(20,6) & {\rm inn}(Q(40,20)) & \cong & Q(20,5) \\ 
{\rm inn}(Q(42,12)) & \cong & Q(21,6) & {\rm inn}(Q(42,21)) & \cong & Q(21,9) \\ 

\end{array}
$$
\caption{The image under {\rm inn} of non-faithful, conjugation Rig quandles} \label{table:innconj}
\end{table}

\begin{remark} 
{\rm
Some of the quandles in Table~\ref{table:conj} are missed in Corollary~\ref{cor:rig} since 
they are homomorphic images under ${\rm inn}$ of quandles that are too large to be found easily  with random searches. 
We were able to identify these using Theorem~\ref{thm:VC}.

}
\end{remark}

\begin{question} Is there a connected quandle $Q$ with a 2-cocycle $\phi: Q\times Q \rightarrow A$ with $|A| > 2$ such that $E(Q,\phi,A)$ is
a conjugation quandle?
\end{question}

\section{Relations among some types of epimorphisms}\label{sec:rel}

In this section, we examine relations among the epimorphisms used in the proofs of the main results;
coverings, abelian extensions and inner representations. 
As we pointed out, abelian extensions and inner representations are coverings.
We show that there is no other implication:

\begin{proposition}
{\rm (i)} There exists an abelian extension that is not an inner representation.

{\rm (ii)} There exists an inner representation that is not an abelian extension.

{\rm (iii)} There exists a covering that is neither an abelian extension nor an inner representation.
\end{proposition}

\begin{proof}
(i) Suppose $X$ is faithful, and let $E=E(X, A)$ be an abelian extension.
Then $R_{(x,a)}=R_{(y,b)}$ if and only if $x=y$, where $(x,a), (y,b) \in E=X \times A$. 
Hence ${\rm inn}: E \rightarrow {\rm inn}(E)$ is an abelian extension, where ${\rm inn}(E)$ is isomorphic to $X$.
If  there is a proper non-trivial subgroup $C$ in $A$, then in \cite{CSV} it was proved that there is a sequence of 
abelian extensions 
$E(X, A) \rightarrow E(X, A/C) \rightarrow X $. 
If $X$ is faithful, therefore, $E(X, A) \rightarrow E(X, A/C)$ is not the inner representation. 
An example is given by $\pi: Q(24,2) \rightarrow Q(12,2)$ where $Q(24,2) = E(Q(6,2),\Z_4)$ and $Q(24,2) = E(Q(12,2),\Z_2)$ as noted in \cite{CSV}.

(ii) This was given in \cite{CSV}: there it is shown that  ${\rm inn}: Q(30, 4) \rightarrow Q(10,1)$ is not an abelian extension since $H_2^Q(Q(10,1))\cong \Z_2$. Another proof that this is not an abelian extension follows from Lemma \ref{lem:criterion} below.

(iii) 
Consider  the product of the two  mappings mentioned above:
$$ {\rm inn} \times \pi: Q(30,4) \times Q(24,2) \rightarrow Q(10,1) \times Q(12,2),$$
where ${\rm inn} \times \pi:  (x,y) \mapsto ({\rm inn}(x), \pi(y))$. We know that ${\rm inn}: Q(30,4) \rightarrow Q(10,1)$
is an inner representation that is not an  abelian extension and $\pi: Q(24,2) \rightarrow Q(12,2)$ is an abelian extension that is not an inner representation. 
We claim that ${\rm inn} \times \pi$ is a covering but not an abelian extension and not an inner representation.

To see that  ${\rm inn} \times \pi$ is not an inner representation note that if ${\rm inn}_Z$ denotes the inner representation on  a quandle $Z$, then 
$${\rm inn}_{X \times Y} = {\rm inn}_X \times {\rm inn}_Y.$$
Thus since $\pi$ is not an inner representation neither is ${\rm inn} \times \pi$.

Since both ${\rm inn}$ and $\pi$ are coverings it is easy to see that their product is a covering. It is more difficult to show that ${\rm inn} \times \pi$ is not an abelian extension. For this we use the following lemma.

\begin{lemma}\label{lem:criterion}  Let $X$ and $Y$ be finite quandles.
Assume $f : Y \rightarrow X$ is an abelian extension, $x \in X$,  $y \in f^{-1}(x)$ and  $\beta \in {\rm Inn}(Y)$.
If $\beta(y) = y$ then $\beta$ acts as the identity on the fiber $f^{-1}(x)$.
\end{lemma}

\begin{proof}
We take $Y = E(X,A,\phi) = X \times A$, and  let $\pi$ be the projection onto $X$.
To simplify notation we write the quandle products in $X$ and $Y$ by juxtaposition. Moreover,
we define $y_1y_2y_3 = (y_1y_2)y_3$ and inductively, for $n > 3$ we set
$y_1y_2 \cdots y_{n+1} =( y_1y_2 \cdots y_{n})y_{n+1}$. 

Since $Y$ is finite for $\beta \in {\rm  Inn}(Y)$ we may write 
 $$\beta  = R_{(x_n,a_n)} R_{(x_{n-1},a_{n-2})} \cdots  R_{(x_1,a_1)}.$$
Then we have for $(z,b) \in X \times A$
$$\beta(z,b) = (z,b)(x_1,a_1)(x_2,a_2) \cdots (x_n,a_n) . $$
Using the definition of the product in $E(X,A,\phi)$ it follows that
$$\beta(z,b) = (zx_1x_2 \cdots x_n, b \, \phi(z,x_1) \, \phi(zx_1,x_2) \, \cdots \, \phi(zx_1x_2\cdots x_{n-1},x_n)).$$
If we assume that $\beta(z,b) = (z,b)$ then we have $$zx_1x_2 \cdots x_n = z$$ and
$$\phi(z,x_1) \, \phi(zx_1,x_2) \, \cdots \, \phi(zx_1x_2\cdots x_{n-1},x_n) = 1.$$
The fiber $\pi^{-1}(z) $ is equal to $ \{(z,c): c \in A \}.$
It follows that $\beta(z,c) = (z,c)$ for all $c \in A$, that is, $\beta$ is  the identity on the
fiber $\pi^{-1}(z)$, as claimed.
\end{proof}

Now returning to the claim that 
$$ {\rm inn} \times \pi: Q(30,4) \times Q(24,2) \rightarrow Q(10,1) \times Q(12,2)$$
is not an abelian extension, we note first that computation shows that for
 ${\rm inn}: Q(30,4) \rightarrow Q(10,1)$ there exist $z \in Q(10,1)$,  $\alpha \in {\rm Inn}(Q(30,4))$ and $x,y \in {\rm inn}^{-1}(z)$
such that $\alpha(x) = x$ but $\alpha(y) \neq y$. This gives an alternative proof via Lemma \ref{lem:criterion} that 
 ${\rm inn}: Q(30,4) \rightarrow Q(10,1)$ is not an abelian extension. Now to show that
$ {\rm inn} \times \pi: Q(30,4) \times Q(24,2) \rightarrow Q(10,1) \times Q(12,2)$ is not an abelian extension, we note that using the same $x$ and $y$ we have that $(x,w)$ and $(y,w)$ for any $w \in Q(24,2)$ lie in the same fiber of ${\rm inn} \times \pi$. It is easy to see that ${\rm Inn}(X \times Y) ={\rm  Inn}(X) \times {\rm Inn}(Y)$ for any quandles $X$ and $Y$. Thus taking $\alpha \in {\rm Inn} (Q(30,4))$ and the identity ${\rm id} \in {\rm Inn}(Q(24,2))$ we have that $\beta = \alpha \times {\rm id} \in {\rm  Inn} (Q(30,4) \times Q(24,2)).$  Clearly 
$\beta(x,w) = (x,w)$ and $\beta(y,w) \neq (y,w)$. So by Lemma \ref{lem:criterion} we have that $ {\rm inn} \times \pi$
is not an abelian extension.
\end{proof}

\subsection*{Acknowledgements}
MS was partially supported by
NIH R01GM109459. The content of
this paper is solely the responsibility of the authors and does not necessarily
represent the official views of  NIH.

\end{document}